\documentclass[12pt]{amsart}

\usepackage[T1]{fontenc}
\usepackage[utf8]{inputenc}
\usepackage{lmodern}
\usepackage{microtype}
\usepackage{amsmath,amssymb,amsthm,mathtools}
\usepackage{enumitem}
\usepackage{booktabs}
\usepackage{hyperref}
\usepackage[nameinlink,capitalize]{cleveref}
\usepackage[margin=3.2cm, marginpar=2.7cm]{geometry}
\usepackage{comment}
\usepackage{xcolor}

\hypersetup{colorlinks=true, linkcolor=red, citecolor=blue, urlcolor=blue}

\newtheorem{theorem}{Theorem}[section]
\newtheorem{proposition}[theorem]{Proposition}
\newtheorem{lemma}[theorem]{Lemma}
\newtheorem{corollary}[theorem]{Corollary}
\theoremstyle{definition}
\newtheorem{definition}[theorem]{Definition}
\theoremstyle{remark}
\newtheorem{remark}[theorem]{Remark}

\newcommand{\del}{\partial}
\newcommand{\Z}{\mathbb{Z}}
\newcommand{\C}{\mathbb{C}}
\newcommand{\R}{\mathbb{R}}
\newcommand{\Sing}{\operatorname{Sing}}
\newcommand{\rk}{\operatorname{rk}}

\newcommand{\CP}{\mathbb{CP}^2}

\DeclareMathOperator{\im}{im}

\title[Topological inequalities for conic-line arrangements]
{Topological inequalities for conic-line arrangements with some prescribed singularities}

\author{Marco Golla and Piotr Pokora}
\address{Department of Mathematics, University of the National Education Commission Krakow, Poland.}
\email{piotr.pokora@uken.krakow.pl}
\address{CNRS and Institut de Math\'ematiques de Marseille, Aix-Marseille Universit\'e, France.}
\email{marco.golla@univ-nantes.fr}

\date{\today}

\begin{document}

\begin{abstract}
We study arrangements of conics and lines in $\mathbb{CP}^2$ from the viewpoint of 4-manifolds. 
We introduce combinatorial conic-line arrangements, the conic-line counterpart of rank-3 matroids, together with their topological and smooth realisations, and we ask which restrictions on a complex arrangement survive when its components are replaced by locally-flat spheres in the same homology classes with the same local singularity models.

For an ADE combinatorial conic-line arrangement of even total degree admitting a topological realisation we prove a Hirzebruch-type inequality;
the proof uses only branched covers and the local topology of rational double points.
In the complex-algebraic category we show that equality holds exactly when the arrangement is a maximising curve, hence free with prescribed exponents.
We then introduce divisible and odd combinatorial conic-line arrangements and obtain inequalities and congruences for the multiplicity sequence of their resolutions, via cyclic branched covers, the $G$-signature theorem, spin structures, and Furuta's $10/8$-theorem.
\end{abstract}

\maketitle

\section{Introduction}

Arrangements of curves in the complex projective plane provide a natural meeting point of algebraic geometry, topology, and combinatorics.

The goal of this paper is to explore to what extent results on arrangements of plane projective curves extend to a more general setup.
We focus mostly (but not solely) on arrangements of lines and conics.

Following earlier work of Ruberman and Starkston~\cite{RL} and of the first author~\cite{G-odd} and Aceto~\cite{AcetoGolla} on line arrangements, we define combinatorial, topological, and smooth conic-line arrangements, and we extend a number of known results in algebraic geometry to this setup.
We will mostly focus on conic-line arrangement with fairly restricted singularities: ADE (or simple) and ordinary multiple points.

We have four main results, each of which is essentially known in the algebro-geometric category or in the case of line arrangements.

\begin{theorem}\label{thm:milnor-bound}
Let $\mathcal A$ be a combinatorial ADE conic-line arrangement of even total degree $D=2m$ that
admits a topological realisation. Then
\begin{equation}\label{eq:milnor-bound}
 \sum_{p\in\Sing(\mathcal A)}\mu_p\leq 3m^2-3m+1,
\end{equation}
where $\mu_p$ is the Milnor number of the singularity at $p$.
\end{theorem}

For conic-line arrangements, we can re-write the inequality in the following way.

\begin{corollary}
    Let $\mathcal{A}$ be an ADE combinatorial conic-line arrangement that admits a topological realisation, with $\deg\mathcal{A}$ even. Let $a_s$ and $d_s$ be the number of singularities of $\mathcal{A}$ of type $A_s$ and $D_s$, respectively.
    Then:
    \[
    a_5 + 2a_7 + d_8 + 2d_{10} + \frac{3\ell}2 \le 2 + a_1 + d_4.
    \]
\end{corollary}

We then introduce divisible and odd combinatorial conic-line arrangements, and we prove two theorems, generalising the results from~\cite{AcetoGolla} to from lines to conics and lines.
Here we state them in the simplified setup of arrangements with only ordinary multiple points.

\begin{theorem}\label{thm:divisible-weak}
    Let $q > 1$ be a prime power and $\mathcal{A}$ a $q$-divisible combinatorial conic-line arrangement with only ordinary multiple points. 
    Call $t_m$ the number of multiple points of multiplicity $m$ for $m \ge 2$.
    If $\mathcal{A}$ is \emph{topologically} realisable, then either $\mathcal{A}$ is a pencil or
    \[
    \sum_{m \ge 2}\left(t_m-\frac{6q}{q+1}\right)\leq 4k+\ell.
    \]
\end{theorem}

\begin{theorem}\label{thm:odd-weak}
    Let $\mathcal{A}$ be an odd combinatorial conic-line arrangement with only with only ordinary multiple points.
    Call $t_m$ the number of multiple points of multiplicity $m$ for $m \ge 2$.
    If $\mathcal{A}$ is \emph{topologically} realisable, then either $\mathcal{A}$ is a pencil or
    \[
    4k + \ell -1 \equiv \sum_j (t_m-1) \pmod{16} \quad \text{ and }\quad
    \sum_j (t_m-9) \le 4k + \ell - 9.
    \]
\end{theorem}

Besides the specific statements, the main goal of the paper is to further strengthen the bridge between 4-dimensional topology and the study of complex curves in complex surfaces.
We collect a number of results that are probably known to experts, or that are well known in one of the two communities but not so much in the other.

For instance, we spend a fair amount of time on the topology of ADE singularities in Section~\ref{s:top-ADE}, and on resolutions of 2-suspension singularities in Section~\ref{s:suspension}.

We also take some time to define a conic-line generalisation of rank-3 matroids, which we call \emph{combinatorial conic-line arrangements}.
However, much like rank-3 matroids, these do not all arise as the underlying combinatorial structure of a complex conic-line arrangement. (In fact, line arrangements/rank-3 matroids already suffice to show this.)

\subsection{History, motivation, and perspectives}
The motivation for this paper comes mostly from two results due to Hirzebruch, which give restrictions on the singularities of a complex curve in two cases: curves with ADE singularities~\cite{Hirzebruch1} and line arrangements~\cite{Hirzebruch}.

In~\cite[Section~5]{Hirzebruch1}, Hirzebruch proves Theorem~\ref{thm:milnor-bound2} in the algebro-geometric category (see also~\cite{Persson}).
In a sense, our proof follows his rather closely, but we are careful to recast the proof in purely topological terms.

In~\cite{Hirzebruch}, Hirzebruch proves the following beautiful inequality for line arrangements in the complex projective planes:
\begin{equation}\label{eq:Hirzebruch-OG}
 t_2 + t_3 \ge \ell + \sum_{m \ge 5}(m-4)t_m,
\end{equation}
where $\ell$ is the number of lines (which is also the degree) and $t_m$ is the number of points of multiplicity $m$ in the arrangement, provided that $t_\ell = t_{\ell-1} = 0$.
Hirzebruch's proof is intrinsically algebro-geometric: one passes to suitable branched abelian covers and applies the Bogomolov--Miyaoka--Yau inequality.
This circle of ideas has subsequently produced a large family of inequalities for arrangements of curves and, more generally, for singular plane curves.
For instance, for line arrangements with $t_m = 0$ for $m \ge \frac{2\ell}3$, Bojanowski~\cite{Bojanowski} proved a much stronger version of~\eqref{eq:Hirzebruch-OG}, which features a \emph{quadratic} term in $m$ on the right-hand side, using Langer's version of the orbifold Bogomolov--Miyaoka--Yau inequality \cite{Langer}.

Both theorems of Hirzebruch have been studied in more details for conic-line arrangements, and especially in the case of maximising, by the second author~\cite{Pokora-conicline-BMY} and Dimca~\cite{DimcaPokoraMaximizing}.
We will give more details on the algebro-geometric aspects of this class of configurations in Section~\ref{ss:AG}.

It is natural to ask to what extent such restrictions depend on the complex-algebraic nature of the arrangement. 
More precisely, suppose that the components of an arrangement are replaced by embedded surfaces in $\mathbb{CP}^{2}$, representing the same homology classes and having the same prescribed local singularity models. 
Which numerical inequalities continue to hold? 

As mentioned above, for line arrangements this question was studied by Ruberman and Starkston~\cite{RL}, and more recently by Aceto and the first author~\cite{AcetoGolla} and by Aceto, McCoy, and Park~\cite{AcetoMcCoyPark}.

\subsection*{Organisation of the paper}
In Section~\ref{s:conic-line-matroids} we
introduce combinatorial conic-line arrangements, their realisations, and their multiplicity multisequences.
In Section~\ref{s:top-ADE} we survey known results on the topology of ADE singularities and study  realisations of ADE combinatorial conic-line arrangements and, more generally, of topological ADE curve configurations, thus proving Theorem~\ref{thm:milnor-bound}.
In the same section we also discuss some algebro-geometric reflections of our results.
In Section~\ref{s:divisible} we introduce and study divisible line arrangements and, using branched covers, we prove a suitable generalisation of Theorem~\ref{thm:divisible-weak}. Section~\ref{s:odd} concerns odd
conic-line arrangements, which we define before proving Theorem~\ref{thm:odd-weak}, using spin structures.
In Section~\ref{s:bonus} we consider configurations that have both ADE singularities and ordinary multiple points, and we prove an inequality that in some sense interpolates between Theorems~\ref{thm:milnor-bound}, \ref{thm:divisible-weak} (when $q=2$), and~\ref{thm:odd-weak} (but which is sensibly weaker than the latter).
Finally, in Section~\ref{s:suspension} we give a topological perspective on the resolution of the double suspension of a curve singularity.

\section{Combinatorial conic-line arrangements and their realisations}\label{s:conic-line-matroids}

In this section we define combinatorial conic-line arrangements, which are to conic-line arrangements what rank-3 matroids are to line arrangements.
Note that there is a much more general concept in the literature, that of the combinatorial type of a curve arrangement~\cite[Remark 3]{Artaletal-Survey}, but this dramatically simplifies when talking about conic-line arrangements.

\begin{definition}
    A \emph{combinatorial conic-line arrangement} $\mathcal{A} = (\mathcal{K}, \mathcal{L}, \mathcal{P}, I)$ is the datum of three sets, $\mathcal{K}$, $\mathcal{L}$, and $\mathcal{P}$, together with a function $I \colon \mathcal{C}\times \mathcal{C} \times \mathcal{P} \to \mathbb{Z}_\ge 0$, where $\mathcal{C} := \mathcal{K} \sqcup \mathcal{L}$, such that:
    \begin{itemize}
        \item $I(C,C,p) = 0$ for all $c\in\mathcal{C}$, $p \in \mathcal{P}$;
        \item $I(C,C',p) = I(C',C,p)$ for all $C,C' \in \mathcal{C}$ and $p \in \mathcal{P}$;
        \item if $I(C,C',p), I(C,C'',p) \ge t$, then $I(C',C'',p) \ge t$;
        \item $\sum_p I(C,C',p) = 4$ for each $C \neq C' \in \mathcal{K}$;
        \item $\sum_p I(C,C',p) = 2$ for each $C \in \mathcal{K}$, $C' \in \mathcal{L}$;
        \item $\sum_p I(C,C',p) = 1$ for each $C\neq C' \in \mathcal{L}$.
    \end{itemize}
    Furthermore, we impose that for each $p\in\mathcal{P}$ the function $I(\cdot,\cdot,p)$ is not identically $0$.

    We call $\mathcal{C}$, $\mathcal{K}$, and $\mathcal{L}$ the (\emph{irreducible}) \emph{components}, the \emph{conics}, and the \emph{lines} of the arrangement, respectively. We let $k = \#\mathcal{K}$, $\ell = \#\mathcal{L}$, $r = k + \ell$, and we call $2k+\ell$ the \emph{degree} of the arrangement.
\end{definition}

\begin{definition}
    A combinatorial conic-line arrangement is:
    \begin{itemize}
        \item a \emph{pencil of lines} if $k = 0$ and $\#\mathcal{P} =1$,
        \item a \emph{pencil of conics} if $\ell = 0$ and $I(C,C',p) = I(D,D',p)$ for every $C\neq C', D \neq D' \in \mathcal{K}$ and every $p \in \mathcal{P}$.
    \end{itemize}
\end{definition}

If $C_1 \cup \dots \cup C_r$ is a conic-line arrangement in $\mathbb{CP}^2$ comprising $k$ conics $C_1, \dots, C_k$ and $\ell$ lines $C_{k+1}, \dots, C_r$, then one easily checks that the incidence data of its components gives a combinatorial conic-line arrangement.

Given $p \in \mathcal{P}$ and $C \in \mathcal{C}$, we say that $p$ \emph{belongs to} $C$, or that $C$ \emph{passes through} $p$, if $I(C,\cdot,p) \not\equiv 0$.
Two curves $C, C'$ passing through $p$ are said to be \emph{transverse} if $I(C,C',p) = 1$, \emph{simply tangent} if $I(C,C',p) = 2$, and \emph{tangent to order} $t$ if $I(C,C',p) = t$.
If $I(C,C',p) > 0$, we also say that $C$ and $C'$ \emph{meet at} $p$.
If $p \in \mathcal{P}$ has $\sum_{C,C'} I(C,C',p) = 2$, i.e. if only two curves meet at $p$ and they do so transversely, we say that $p$ is a \emph{node}.

\begin{definition}
    A combinatorial conic-line arrangement $\mathcal{A}$ is \emph{realisable over} $\mathbb{K}$ if there exists a conic-line arrangement whose incidence data are encoded by $\mathcal{A}$.
    If $\mathbb{K} = \mathbb{C}$, we simply say that it is \emph{realisable}.
\end{definition}

We will be interested in \emph{symplectic}, \emph{smooth} and \emph{topological} realisations of $\mathcal{A}$ in $\CP$.

First, recall that $\CP$ supports a unique symplectic structure $\omega_{\rm FS} \in \Omega^2(\CP)$, called the \emph{Fubini--Study} structure, and that $\omega_{\rm FS}$ is compatible with the complex structure on $\CP$. A smoothly embedded surface in $\CP$ is said to be \emph{symplectic} if its tangent planes are $\omega_{\rm FS}$-positive. Equivalently, there exists an almost-complex structure $J$, compatible with $\omega_{\rm FS}$, such that the surface is the image of a $J$-holomorphic map.

\begin{definition}
    A combinatorial line arrangement $\mathcal{A}$ is \emph{symplectically realisable} if there exists an almost-complex structure $J$ and a collection of $J$-holomorphic curves realising the combinatorics of $\mathcal{A}$.
\end{definition}

Let us now move to the smooth and topological categories.
Recall that the incidence data of a collection of non-singular complex curves at a point determines the topological type of a singularity, which in turn can be encoded in its \emph{link}, which is a link in the 3-sphere.
At each point $p$, a complex curve (regardless of whether its components are non-singular) is locally homeomorphic to the cone over its link at $p$.
The links arising as links of singularities (regarless of whether the components are singular or not) in these fashion are called \emph{algebraic links}.
Incidentally, an algebraic link comes from non-singular local branches   if and only if all of its components are unknots.

The multiplicity of intersections of two components at a point can be read off from the linking number of the corresponding components of the link of its singularity.

\begin{definition}
    A combinatorial line arrangement $\mathcal{A}$ is \emph{smoothly realisable} (respectively, \emph{topologically realisable}) if there exists a collection of smoothly embedded (resp., locally-flatly embedded), oriented 2-spheres in $\CP$ such that:
    \begin{itemize}
        \item at each intersection point of at least two spheres, the union of the spheres is \emph{homeomorphic} to the cone over an algebraic link;
        \item the collection of linking numbers of these algebraic links realises the combinatorial type of $\mathcal{A}$.
    \end{itemize}
\end{definition}

To be slightly more precise: every component of the link of a singularity is oriented and comes from a single sphere. The pairwise linking numbers at each point give a function from pairs of components to the integers at each point, which provide the incidence functions.

Note that a smooth realisation is in particular a topological realisation.

The main question we want to address in the paper is the following. Given a combinatorial line arrangement $\mathcal{A}$, does it admit a smooth or topological realisation?

In what follows, we write $h\in H_2(\mathbb{CP}^2;\Z)$ for the positive generator, i.e. the homology class of a complex line.

Let
\[
 C = C_1\cup\ldots\cup C_r\subset \mathbb{CP}^2
\]
be a topological realisation of a conic-line arrangement consisting of $k$ conics and $\ell$ lines, as above.
From the assumptions on the incidence function, possibly up to a global orientation-reversal of all of the components of $C$, each line of $C$ represents the homology class $h$ and each conic of $C$ represents $2h$.

\subsection{The multiplicity multisequence}\label{ss:multisequence}

Suppose that $C$ is an algebraic curve in $\CP$.
Then there exists a minimal sequence of blow-ups of $\CP$, yielding a complex surface $Y$ that is diffeomorphic to
\[
Y \cong \mathbb{CP}^{2}\#N\overline{\mathbb{CP}}^{\,2}
\]
such that the \emph{strict transform} of $C$ is non-singular.
(The strict transform is obtained from the preimage of $C$ by the blow-ups by first removing the exceptional divisor and then completing. It is biholomorphic to the normalisation of $C$.)

This sequence of blow-ups comes with a collection of \emph{multiplicities}, which keep track of the multiplicities of the singularities of $C$ when we keep blowing up:
\[
m_1,\ldots,m_N.
\]
Since this collection of multiplicities depends only on the collection of the singularities, it makes sense to extend the definition to a general abstract curve arrangement (which we have not defined).

Let us give a definition that is suited for combinatorial conic-line arrangements.
Let $\mathcal{A} = (\mathcal{K}, \mathcal{L},\mathcal{P},I)$ be a combinatorial conic-line arrangement.
First, note that for each integer $t$ (which we think of as an order of tangency) we have an equivalence relation $\sim_t$ on the set of pairs $(C,p)$, where $C \in \mathcal{C} = \mathcal{K} \cup \mathcal{L}$, given by:
\[
(C,p) \sim_t (C',p') \quad \Longleftrightarrow \quad p = p' \text{ and } I(C,C',p) \ge t.
\]
These equivalence relation partitions $\mathcal C \times \mathcal{P}$ into equivalence classes, and we denote $(m^t_j)$ unordered sequence of the sizes of these equivalence classes \emph{when they contain at least two elements}. (That is to say, we disregard singletons, and we just keep the cardinalities of the other classes.)
\begin{definition}
    The \emph{multiplicity multisequence} of the combinatorial conic-line arrangement $\mathcal{A}$ is un-ordered sequence $(m_j) = (m^t_j)_{t = 1,\dots,4}$.
\end{definition}

We invite the reader to check that if $C$ is a complex realisation of $\mathcal{A}$, then the two definitions of multiplicity multisequence do indeed agree.

\begin{lemma}\label{l:selfintersection}
In the notation above,
\begin{equation}\label{eq:noadjunction}
(2k+\ell)^2 - \sum_j m_j^2 = 4k + \ell - \sum_j m_j.    
\end{equation}
\end{lemma}

The lemma can be proven in the algebro-geometric category by using the adjunction formula.
We sketch a combinatorial proof, leaving some elementary checks for the reader.

\begin{proof}[Sketch of proof]
    We will work by induction on the quantity:
    \[
    M(\mathcal{A}) = 2k^2-2k+2k\ell+\frac12(\ell^2-\ell) - a_1,
    \]
    where $a_1$ is the number of nodes of $\mathcal{A}$. (The reason why this is quantity is called $a_1$ will be clear in the next section.)

    First, note that, for any conic-line arrangement,
    \[
    \sum_{p, C, C'} I(C,C',p) = 4\binom{k}{2} + 2k\ell + \binom{\ell}{2}.
    \]
    In particular, a conic line arrangement has at most $2k^2-2k+2k\ell+\frac12(\ell^2-\ell)$ nodes, and if it does \emph{all} singular points are nodes.
    That is to say, $M(\mathcal{A}) = 0$ if and only if each point $p \in \mathcal{P}$ is a node.
    In this case, equation~\eqref{eq:noadjunction} is immediate to verify, as $m_j = 2$ for each $j$, and there are $a_1$ entries in the multiplicity multi-sequence.

    The inductive step works as follows.
    Given $\mathcal{A}$ for which $M(\mathcal{A}) > 0$, choose $p \in \mathcal{P}$ that is not a node and $C \in \mathcal{C}$ passing through $p$.
    Define $\mathcal{A}'$ as follows:
    \begin{align*}
        \mathcal{K}' &= \mathcal{K}, \quad \mathcal{L}' = \mathcal{L},\\
        \mathcal{P}' &= \mathcal{P} \sqcup \{p_{i,t} \mid i = 1,\dots,r, 1 \le t \le I(C,C_i)\},
    \end{align*}
    and let $I'(\cdot,\cdot,q) = I(\cdot, \cdot, q)$ for $q \in \mathcal{P}\setminus \{p\}$, let $p_{i,t}$ be a node where $C_i$ and $C$ meet transversely for $t = 1, \dots, I(C,C_i,p)$, and
    \[
        I'(C_i,C_j,p) = \left\{\begin{array}{ll}
            I(C_i,C_j,p) & \text{if } C_i, C_j \neq C\\
            0 & \text{if } C_i = C \text{ or } C_j = C.
        \end{array}
        \right.
    \]
    (This is the combinatorial version of locally perturbing $C$ to intersect transversely all other components in a neighbourhood of $p$.)
    It is easy to verify that $\mathcal{A}'$ is still a conic-line arrangement, and that $M(\mathcal{A}') < M(\mathcal{A})$, thus completing the inductive step.
\end{proof}

\section{Topological ADE realisations}\label{s:top-ADE}

In this section, we restrict to configurations with \emph{simple} or \emph{ADE} singularities.
We will look at a more general situation, and we will prove Theorem~\ref{thm:milnor-bound2} below, which directly implies Theorem~\ref{thm:milnor-bound}.

\subsection{ADE singularities}\label{ss:ADE}
In the case of conic-line configurations, only the following ADE singularities can arise:
\begin{itemize}
    \item $A_{2t-1}$, $1 \le t \le 4$: when exactly two curves meet at a point, and their order of tangency is $t$;
    \item $D_{2t+2}$, $1 \le t \le 4$: when exactly three curves meet at a point, two of them having order of tangency $t$ and the other one being transverse to both.
\end{itemize}
For instance, $A_1$ is an ordinary double point (a node), $A_3$ is a simple tangency (a tacnode), and $D_4$ is an ordinary triple point.

Note that a simple singularity in a conic-line arrangement can be read off from the incidence matrix, so it is a combinatorial property.

\begin{definition}
    A combinatorial conic-line arrangement is of \emph{type ADE} if all of its singularities are of type ADE.
\end{definition}



Recall the statement that we want to prove:
if a combinatorial ADE conic-line arrangement $\mathcal A$ has even total degree $D=2m$ and
admits a topological realisation, then
\[
 \sum_{p\in\Sing(\mathcal A)}\mu_p\leq 3m^2-3m+1.
\]
Here, $\mu_p$ is the \emph{Milnor number} of the singularity.
This has both an algebro-geometric and a topological interpretation.

Algebro-geometrically, if $C = V(f) \subset \C^2$ has a singularity of type $A_s$, $D_s$, or $E_s$ at the point $0 \in \C^2$, then
\[
\mu_{0} = \dim \C[[x,y]]/(\partial_x f, \partial_y f).
\]
Topologically, if $C' = V(f+\varepsilon)$ for some $|\epsilon| \ll 1$, then $M = C' \cap B^4_\eta$ is a smoothly embedded surface with boundary, called the \emph{Milnor fibre} of the singularity of $C$ at $0$, and
\[
\mu_{0} = b_1(M) = \dim H_1(M;\C) = \rk H_1(M;\Z).
\]
Note that $M$ bounds a link which is isotopic to $C' \cap \del B^4_\eta$, and $\mu$ turns out to be the minimal $b_1$ of any \emph{smoothly} embedded surface in the 4-ball (or even in the 3-sphere) that bounds the same link~\cite{KronheimerMrowka}.
This is also true in the topological category for ADE singularities, but \emph{not} for more general algebraic knots~\cite{BaaderFellerLewarkLiechti}.

In each category among topological, smooth, and symplectic, one can individually replace a singularity of a curve arrangement by its Milnor fibre, and obtain a curve arrangement in the same category, just one with one less singularity. (This is \emph{not} always possible in the algebro-geometric category!)
We will refer to such an operation as a \emph{smoothing}.

For ADE singularities, something more is true: $-\mu$ is the Murasugi--Trotter signature of the link of the singularity. To us, this implies the following: if $F$ is any properly, locally-flatly embedded surface in $B^4_\eta$ whose boundary is the same as that of $M$, then the intersection form of the double-cover $\Sigma(B^4_\eta,F)$ of $B^4_\eta$ ramified over $F$ has signature $-\mu$.
We will not need it here, but $\Sigma(B^4_\eta,F)$ is in fact diffeomorphic to the total space of the resolution of the surface ADE singularity that suspends the singularity of $C$, and its intersection form is isomorphic to the corresponding ADE lattice.
See, for example, \cite{Brieskorn,Durfee}, or Section~\ref{s:suspension} below.

Since $b_2(\Sigma(B^4_\eta,M)) = b_1(M) = \mu$, this in turn implies that $\Sigma(B^4_\eta,M)$ is negative definite.

Incidentally, $\Sigma(B^4_\eta,M)$ is the Milnor fibre of the \emph{(double) suspension} of the singularity of $C$ at $0$, where the $n$-suspension is defined by the equation $z^n+f(x,y)$ and gives an isolated hypersurface singularity in $\C^3$.
For an ADE singularity, the Milnor fibre is diffeomorphic to the corresponding ADE plumbing of cotangent bundles of 2-spheres (with the opposite orientation---this is the orientation convention in for cotangent bundles in symplectic topology).

Given a plane curve $C$, we will denote with $\mu(C)$ the \emph{total Milnor number} of $C$, which is simply the sum of the Milnor numbers of its singularities.

\subsection{Topological ADE arrangements of curves}

We will prove something stronger than Theorem~\ref{thm:milnor-bound}, and we will fully recover Hirzebruch's inequality.
We need to set up one more definition, first.

\begin{definition}
    A \emph{smooth} (respectively, \emph{topological}) \emph{arrangement of curves} is the image $C = \im(f)$ of a map $f \colon \sqcup_{i=1}^r F_i \to \CP$ such that:
    \begin{itemize}
        \item $F_i$ is a homologically essential, closed, oriented surface for each $i$;
        \item $f$ is injective and smooth (resp., locally flat) away from finitely many points $q_1, \dots, q_n \in \sqcup_i F_i$;
        \item at each point $p = f(q_i)$ has a neighbourhood $p \in U \subset \CP$ such that either $C\cap U$ is locally-flatly embedded or $C \cap U$ is homeomorphic to the neighbourhood of a complex curve singularity as \emph{oriented} surfaces in $B^4$; in this case, $p$ has a well-defined Milnor number, $\mu_p$;
        \item $C$ satisfies:
        \[
            \sum_i \chi(F_i) - n + \sum_p (1-\mu_p) \ge 3D-D^2,
        \]
        where $D$ is the degree of $C$, defined below.
    \end{itemize}
\end{definition}
    Note that, up to a global orientation sign, $f_*(F_i) = D_i h$ for some $D_i > 0$. We call $D_i$ the \emph{degree} of $f(F_i)$ and $D = \sum_i D_i$ the \emph{degree} of the arrangement. We will call either of $f$ and $\sqcup F_i$ the \emph{normalisation} of $C$.
    
\begin{definition}
We say that $C$ is:
\begin{itemize}
\item  of \emph{ADE type} if all singularities of $C$ are of type ADE.
\item  of \emph{ADE-transverse type} if all singularities of $C$ are either of type ADE or they are ordinary multiple points.
\end{itemize}
\end{definition}

    We can extend the definition of total Milnor number from complex curves to topological and smooth arrangements of curves, by summing the Milnor numbers of all the local singularity types:
    \[
    \mu(C) := \sum_p \mu_p.
    \]
    Likewise, we extend the definition of \emph{multiplicity multisequence} by concatenating the multiplicity sequences of all the singular points of $C$.

The main theorem of this section is then the following.

\begin{theorem}\label{thm:milnor-bound2}
Let $C \subset \CP$ be a topological ADE arrangement of curves of even total degree $D=2m$.
Then
\begin{equation}\label{eq:milnor-bound2}
 \mu(C) \leq 3m^2-3m+1.
\end{equation}
\end{theorem}

As mentioned above, there is a more general notion of combinatorial curve arrangement, as in~\cite{Artaletal-Survey}, as well as of their complex, symplectic, smooth, and topological realisations.
A topological ADE arrangement in the sense of the definition above would then be a topological realisation of a combinatorial arrangement whose singularities are all of type ADE.

A topological realisation of an ADE conic-line arrangement is in particular a topological ADE arrangement.
The only point which requires some justification is the condition on the Euler characteristic, but we leave it to the reader to verify that this is satisfied.



For later use, we recall some local data of ADE singularities in Table~\ref{table:ADE}.
\begin{center}
\footnotesize
\begin{table}[h!]
\begin{tabular}{c@{\qquad}c@{\qquad}c}
singularity & Milnor number & number of branches\\
\midrule
$A_{2s+1}$ & $2s+1$ & 2\\
$A_{2s}$ & $2s$ & 1\\
$D_{2s+2}$ & $2s+2$ & 3\\
$D_{2s+1}$ & $2s+1$ & 2\\
$E_6$ & 6 & 1\\
$E_7$ & 7 & 2\\
$E_8$ & 8 & 1\\
\phantom{a}
\end{tabular}
\caption{Local data of ADE singularities.}\label{table:ADE}
\end{table}
\end{center}


To prove Theorem~\ref{thm:milnor-bound2}, we begin with a standard observation.
In the case of symplectic curve arrangements, this is well-known (and equality holds instead of an inequality).

\begin{lemma}\label{lem:smoothing-genus}
Let $C \subset \CP$ be a topological ADE arrangement of degree $D$. Let $F$ be obtained by simultaneously smoothing of all singular points of $C$.
Then $F$ is connected,
\[
 [F]=Dh,
 \qquad
 g(F) \le \frac{(D-1)(D-2)}{2}.
\]
Moreover, if $C$ was a smooth arrangement, then equality holds.
\end{lemma}

\begin{proof}
Let $f \colon \sqcup_i F_i \to \CP$ be normalisation of $C$.

Choose pairwise disjoint Milnor balls around the singular points and replace the singular local branches by their Milnor fibres.
Since all components have non-zero homology class, every two components intersect.
Since each Milnor fibre is connected, replacing a singularity by its Milnor fibre joins all the components meeting at that point.
This proves that $F$ is connected.

Since $F$ is obtained from $C$ by doing embedded surgery in some 4-balls, the homology class $[F]$ is the same as the sum of the homology classes $f_*([F_i])$, so it is indeed $Dh$.

To compute the genus of $F$, we perform an Euler characteristic computation.

If the singularity of $C$ at $p$ has $\beta_p$ branches (see Table~\ref{table:ADE}), then removing a Milnor ball from $C$ removes three discs from $\sqcup_i F_i$.
Looking at all the singularities of $p$, these discs we are removing are centred at each of the points $q_1, \dots, q_n$, which are the non--locally-flat points of $C$.
Call $F_i'$ the surface $F_i$ with these discs removed.
The Euler characteristic of $\sqcup_i F'_i$ is therefore $\sum_i \chi(F_i) - n$.

By additivity of the Euler characteristic, gluing in the Milnor fibre $M_p$ at $p$ adds $\chi(M_p)$, which is $1-\mu_p$, so
\[
\chi(F) =  \sum_i \chi(F_i) - n + \sum_p (1-\mu_p) \ge 3D-D^2,
\]
by definition of topological ADE arrangement.

Since $F$ is connected, $g(F) \le \frac{(D-1)(D-2)}2$.

Finally, if $C$ were a smooth arrangement, then $g(F) \ge \frac{(D-1)(D-2)}2$~\cite{KronheimerMrowka}, which proves the equality.
\end{proof}

From now on, let us assume that $C$ is a topological ADE arrangement of even degree: $D=2m$ is even.

Let
\[
 \pi\colon X\longrightarrow\mathbb{CP}^2
\]
be the double cover of $\mathbb{CP}^2$ branched over the surface $F$ supplied by Lemma~\ref{lem:smoothing-genus}.

The following lemma is well-known in the algebraic setting, but our aim here is to deliver a concise complete topological proof.

\begin{lemma}[Hirzebruch]
\label{lem:betti-signature}
The branched cover $X$ satisfies
\[
 b_1(X)=0,
 \qquad
 b_2(X) \le D^2-3D+4,
 \qquad
 \sigma(X)=2-\frac{D^2}{2}.
\]
Consequently,
\begin{equation}\label{eq:bminus-double-plane}
 b_2^-(X) \le \frac34D^2-\frac32D+1=3m^2-3m+1.
\end{equation}
\end{lemma}

\begin{proof}
The Euler characteristic formula for a double branched cover gives
\[
 \chi(X)=2\chi(\mathbb{CP}^2)-\chi(F) = 6-(2-2g(F)).
\]
Using Lemma \ref{lem:smoothing-genus}, this becomes
\[
 \chi(X) \le D^2-3D+6.
\]

We next explain the vanishing of $b_1(X)$. Since $[F]=Dh$ is divisible by $2$,
Poincar\'e--Lefschetz duality and the long exact sequence of the pair
$(\mathbb{CP}^2,F)$ give
\[
 H_1(\mathbb{CP}^2\setminus F;\Z/2)\cong\Z/2.
\]
By the Goldschmidt lemma~\cite{HsiangSzczarba} (see also \cite[Lemma 3.3]{AcetoGolla}; alternatively, one can use~\cite[Section~1]{Gilmer}) $b_1(X)=0$.
Hence $b_2(X) = \chi(X)-2 \le D^2-3D+4$.

We now apply the G-signature theorem \cite{AtiyahSinger} which gives us
\[
 \sigma(X)=2\sigma(\mathbb{CP}^2)-\frac12[F]^2
 =2-\frac{D^2}{2}.
\]
Note that, while the original proof of the G-signature theorem only works for smoothly embedded surfaces, the result actually holds for locally-flatly embedded surfaces (see~\cite{RL, AcetoGolla} and references therein).

Formula \eqref{eq:bminus-double-plane} follows from
$2b_2^-(X)=b_2(X)-\sigma(X)$.
\end{proof}

The key observation now is that the local Milnor fibres inject into the global double cover.

\begin{lemma}\label{lem:local-lattices}
For each $p\in\Sing(C)$, let $W_p$ be the part of $X$ lying over a Milnor ball around $p$.
Then $H_2(W_p;\R)$ carries the negative-definite ADE root lattice of rank $\mu_p$, and the natural map
\[
 \bigoplus_{p\in\Sing(C)}H_2(W_p;\R)
 \longrightarrow H_2(X;\R)
\]
is injective. In particular,
\[
 \mu(\mathcal C) \leq b_2^-(X).
\]
\end{lemma}

\begin{proof}
As seen above, locally, the double cover over a smoothing of a plane curve singularity
$f_p(x,y)=0$ is the Milnor fibre of the suspended hypersurface singularity
\[
 z^2+f_p(x,y)=0.
\]
and its Milnor fibre has negative-definite intersection form. In fact, it is the corresponding ADE root lattice of rank $\mu_p$.

The boundaries $\partial W_p$ are links of rational double points and hence are rational homology spheres.
Let \[
W=\bigsqcup_{p\in\operatorname{Sing}(C)}W_p
\qquad \text{ and } \qquad
Z=\overline{X\setminus\operatorname{int}(W)}.
\] 
The Mayer--Vietoris sequence with real coefficients contains
\[
 H_2(\partial W;\R)
 \longrightarrow H_2(W;\R)\oplus H_2(Z;\R)
 \longrightarrow H_2(X;\R)
 \longrightarrow H_1(\partial W;\R).
\]
Both outer terms vanish, so the middle map is injective.
The direct sum of the local negative-definite lattices therefore embeds in the negative part of the global intersection form.
\end{proof}

\begin{proof}[Proof of Theorem~\ref{thm:milnor-bound2}]
Let $X$ be the double cover of $\CP$, ramified over the global smoothing of $C$.
By Lemma~\ref{lem:betti-signature}, $b_2^-(X) \le 3m^2-3m+1$, and by Lemma~\ref{lem:local-lattices} $\sum_p \mu_p \le b_2^-(X)$.
Concatenating the two inequalities, we get the required statement.
\end{proof}

\begin{remark}

For $D=6$, the branched double cover satisfies
\[
b_1(X)=0,\qquad b_2(X)=22,\qquad \sigma(X)=-16.
\]
Thus it has the same rational Betti numbers and signature as a K3 surface. In particular,
\[
b_2^{-}(X)=19,
\]
which is the right-hand side of~\eqref{eq:milnor-bound}.
\end{remark}
\begin{remark}
Hirzebruch's proof of the above inequality in the \emph{algebraic} category is a bit more involved. Let us recall crucial steps of his reasoning.
\begin{enumerate}
\item Hirzebruch worked with a minimal resolution of the singularities of the branched cover $X$ and he used the fact that the number of exceptional curves in the minimal resolution of singularity $W_{m}$ with $W\in \{A,D,E\}$ is $m$, and the union of such curves spans a vector subspace in the N\'eron--Severi group of the resolution.
\item He then applied the topological signature theorem and the Lefschetz theorem on $(1,1)$-classes.
\item Finally, he used Brieskorn’s result on the simultaneous resolution of simple surface singularities that implies that the topology and the Hodge numbers of the minimal resolution of $X$ are the same as those of the double cover of $\mathbb{CP}^{2}$ ramified along a smooth curve of the same degree.
\end{enumerate}
\end{remark}

\subsection{A Hirzebruch-type inequality for ADE conic-line arrangements}

In this section, we rephrase Theorem~\ref{thm:milnor-bound2} for conic-line arrangements.



Recall that the only ADE singularities with smooth branches are $A_{2t-1}$, when two branches have contact of order $t$, and $D_{2t+2}$ which is a singularity of type $A_{2t-1}$ plus a transverse branch.

\begin{theorem}[Topological Hirzebruch-type inequality]\label{thm:hirzebruch-top}
Let $C$ be either a smooth ADE curve arrangement with normalisation $\sqcup F_i$ or a topological realisation of an ADE conic-line of $k$ conics and $\ell$ lines, with $\ell$ even.
Suppose that all singularities of $C$ have non-singular branches, and call $a_{2t-1}$ (respectively, $d_{2t+2}$) the number of singularities of type $A_{2t-1}$ (resp., $D_{2t+2}$).

Then, in the case of a smooth ADE arrangement:
\begin{equation}\label{eq:milnor-bound-smooth-branches}
\sum_{t\ge 3} (t-2)a_{2t-1} + \sum_{t \ge 6} (t-2)d_{2t+2} \le  a_1 + d_4 + 2 + 3m - \sum_i \frac{3\chi(F_i)}2,    
\end{equation}
while in the case of a topological realisation of an ADE conic-line arrangement:
\begin{equation}\label{eq:main-hirzebruch}
a_5 + 2a_7 + d_8 + 2d_{10} + \frac{3\ell}2 \le 2 + a_1 + d_4.
\end{equation}
\end{theorem}

\begin{proof}
%
We place ourself in the same notation as in the previous subsection: $C$ has degree $D = 2m$, $F$ is obtained by simultaneously smoothing all singularities of $C$, and $X$ is the double cover of $\CP$, ramified over $F$.

Note that, both in the case of topological ADE conic-line arrangements and for smooth ADE curve arrangements, equality is attained in the inequality:
\[
\sum_i \chi(F_i) - n + \sum_p (1-\mu_p) \ge 3D - D^2.
\]
On the other hand, when a configuration has only smooth branches, we can compute $n$ above:
\[
n = 2\sum_s a_s + 3\sum_s d_s.
\]
So we obtain:
\[
2m^2 = 3m - \sum_i \frac{\chi(F_i)}2 + \sum_t ta_{2t-1} + \sum_t (t+2)d_{2t+2}.
\]
Then, from~\eqref{eq:milnor-bound}:
\begin{align*}
\sum_t (4t-2)a_{2t-1} &+ \sum_t (4t+4)d_{2t+2} \le 6m^2-6m+2\\
& = 2 + 3m - \sum_i \frac{3\chi(F_i)}2 + \sum_t 3ta_{2t-1} + \sum_t (3t+6)d_{2t+2}.
\end{align*}
Rearranging, we get:
\[
\sum_{t\ge 3} (t-2)a_{2t-1} + \sum_{t \ge 3} (t-2)d_{2t+2} \le  a_1 + d_4 + 2 + 3m - \sum_i \frac{3\chi(F_i)}2,
\]
as required.

Specialising to conic-line arrangements, we have:
\[
m = k+\frac{\ell}2,\qquad \sum_i\frac{\chi(F_i)}2 = k+\ell,
\]
and we only have singularities of type $a_1, \dots, a_7$ and $d_4, \dots, d_{10}$, so
\[
a_5 + 2a_7 + d_8 + 2d_{10} + \frac{3\ell}2 \le 2 + a_1 + d_4,
\]
as we set out to prove.
\end{proof}

\begin{remark}
    In fact, when dealing with topological realisations, one can bump up the assumption to asking that all components have degree at most $4$: for these arrangements, too, the inequality~\eqref{eq:milnor-bound-smooth-branches} holds.
    For those realisations, the locally-flat minimal-genus agrees with the smooth minimal genus, which in turn agrees with the genus of a non-singular complex curve of the main degree.
    This is the key ingredient in the proof to eliminate the quadratic term $m^2$ from the inequality~\eqref{eq:milnor-bound}.
\end{remark}


\begin{remark}
For conic arrangements (no lines as irreducible components), the algebraic inequality \cite[Theorem 3.4]{BromboszczJaroslawskiPokora} detects all of $A_3,A_5,A_7$ and has stronger coefficients. The advantage of our proof of~\eqref{eq:main-hirzebruch} is that it
only uses four-dimensional topology, and therefore applies to positive topological ${\rm ADE}$ realisations. This is very important since we are not using any variation on the Bogomolov--Miyaoka--Yau inequality, including the orbifold one from \cite{Kobayashi}.
\end{remark}

\subsection{An algebro-geometric perspective}\label{ss:AG}

We conclude this section by explaining the meaning of equality in Theorem~\ref{thm:hirzebruch-top}.
In the algebraic setting, equality is equivalent to the curve being \emph{maximising} and hence, by a theorem of Dimca and the second author~\cite{DimcaPokoraMaximizing}, to its freeness.

We need the following definitions.
Denote by $S = \mathbb{C}[x,y,z]$ the coordinate ring of $\mathbb{CP}^{2}$.

\begin{definition}
Let $C = \{f=0\} \subset \mathbb{CP}^{2}$ be a reduced curve of degree ${\rm deg}(f) = D$. Consider the graded $S$-module of Jacobian syzygies of $f$, namely
\[
AR(f)=\{(a,b,c)\in S^{\times 3} : af_x+bf_y+cf_z=0\}.\]
The minimal degree of non-trivial Jacobian relations for $f$ is defined as
\[
{\rm mdr}(f):=\min\{r : AR(f)_r\neq 0\}.
\]
\end{definition}

We now recall the notion of a free plane curve. First we recall the definition of Tjurina number of a reduced plane curve singularity.

Let $(C,0)$ be the curve singularity in $\C^2$ defined by $f(x,y) = 0$.
The \emph{Tjurina number} of $(C,0)$ is
\[
\tau_{0} = \dim_\C \C[[x,y]]/(f,\partial_xf, \partial_y f).
\]
The \emph{Tjurina number} $\tau(C)$ of a plane curve $C$ is the sum of the Tjurina numbers of all of its singularities.

Note that, by definition, $\tau(C) \le \mu(C)$ for every $C$.

\begin{definition}
\label{d2}
Let $C = \{ f=0 \}$ be a reduced curve in $\mathbb{CP}^{2}$ of degree $D$. Then the curve $C$ with $r:={\rm mdr}(f)\leq (D-1)/2$ is \emph{free} if 
\begin{equation}
\label{duPles}
(D-1)^{2} - r(D-r-1) = \tau(C).
\end{equation}
The pair $(d_{1}, d_{2}) = ({\rm mdr}(f),  D -{\rm mdr}(f)-1)$ is called the \emph{pair of exponents} of $C$.
\end{definition}
\begin{proposition}
\label{cor:equality-freeness}
Let $C\subset\mathbb{CP}^{2}$ be a reduced conic-line arrangement of even degree \(D=2m>4\) with only
ADE singularities. Then the following conditions are equivalent:
\begin{enumerate}
    \item equality holds in \eqref{eq:main-hirzebruch}, that is,
    \[
    a_5+2a_7 + d_8+2d_{10}+ \frac{3}{2}\ell = a_{1}+d_{4}+2;
    \]
    \item the total Milnor number of \(C\) is maximal, i.e. equality holds in~\eqref{eq:milnor-bound}:
    \[ \mu(C)=3m^{2}-3m+1;
    \]
    \item the total Tjurina number is maximal, i.e. it satisfies
    \[
    \tau(C)=3m^{2}-3m+1;
    \]
    \item \(C\) is a maximising curve;
    \item \(C\) is a free plane curve with exponents
    \[
    (d_{1},d_{2}) = (m-1,m).
    \]
\end{enumerate}
\end{proposition}

\begin{proof}
The equivalence of \textup{(1)} and \textup{(2)} follows immediately
from the following observation:
\begin{equation}\label{eq:Hirzebruch-Milnor-defect}
a_{1}+d_{4}+2
-
\sum_{t\geq 3}(t-2)a_{2t-1}
-
\sum_{t\geq 3}(t-2)d_{2t+2}
-
\textstyle{\frac{3}{2}}\ell
=
2\left(3m^{2}-3m+1-\mu(C)\right).
\end{equation}

Next, recall that every ADE-singularity is weighted homogeneous and consequently its Milnor and Tjurina numbers coincide, i.e. $\tau_{p}=\mu_{p}$ for every $p\in\operatorname{Sing}(C)$.
Therefore
\[
\tau(C)
=
\sum_{p\in\operatorname{Sing}(C)}\tau_{p}
=
\sum_{p\in\operatorname{Sing}(C)}\mu_{p}
=
\mu(C),
\]
which proves the equivalence of\/ \textup{(2)} and \textup{(3)}.

By definition, a reduced plane curve of even degree $2m>4$ with only
ADE singularities is maximising when
\[
\tau(C)=3m(m-1)+1.
\]
Thus \textup{(3)} and \textup{(4)} are equivalent. Finally, by
\cite[Theorem~2.9]{DimcaPokoraMaximizing}, such a curve is maximising if
and only if it is free with exponents
\[
(m-1,m).
\]
This proves the equivalence of\/ \textup{(4)} and \textup{(5)}.
\end{proof}
\begin{remark}
Under the assumptions of
Proposition~\ref{cor:equality-freeness}, equality in Theorem \ref{thm:hirzebruch-top} is equivalent
to
\[
AR(f)
\simeq
S(-(m-1))\oplus S(-m).
\]
In particular,
\[
\operatorname{mdr}(f)=m-1.
\]
Equivalently, the Milnor algebra
\[
M(f)=S/(f_{x},f_{y},f_{z})
\]
has the minimal graded free resolution
\[
0
\longrightarrow
S(-(3m-2))\oplus S(-(3m-1))
\longrightarrow
S^{3}(-(2m-1))
\longrightarrow
S
\longrightarrow
M(f)
\longrightarrow
0.
\]
\end{remark}
\begin{remark}[Topological meaning of equality]
The numerical equivalence between equality in Theorem~5.2 and
\[
\mu(C)=b_{2}^{-}(X)
\]
continues to hold for a positive topological ADE realisation.
(Here $X$ is the double cover of $\CP$ ramified over a smoothing of $C$, as above.)
In that case,
the direct sum of the local ADE Milnor lattices has rank \(b_{2}^{-}(X)\).
Since the natural map
\[
\bigoplus_{p\in\operatorname{Sing}(C)}
H_{2}(W_{p};\mathbb{R})
\longrightarrow
H_{2}(X;\mathbb{R})
\]
is injective and its image is negative definite, equality implies that this
image is a maximal negative-definite subspace of \(H_{2}(X;\mathbb{R})\).
Equivalently, \textit{its orthogonal complement is positive definite}.

This is the natural four-dimensional topological interpretation of the
equality case. It must, however, be distinguished from freeness. Freeness is
an algebraic property defined in terms of the Jacobian syzygies of a
homogeneous polynomial defining a complex projective curve. There is no
corresponding Jacobian module for an abstract collection of locally flat
spheres in $\mathbb{CP}^{2}$. Thus equality for a purely topological ADE
realisation does not, by itself, define or imply freeness. The freeness
conclusion in Proposition~\ref{cor:equality-freeness} requires \(C\) to be a
reduced complex algebraic curve. 

Moreover, unlike classical Hirzebruch-type inequalities derived from variants of the Bogomolov--Miyaoka--Yau inequality, the topological version can detect homological properties of complex plane curves. This feature constitutes the principal novelty of this approach.
\end{remark}

\begin{remark}
Identity \eqref{eq:Hirzebruch-Milnor-defect} shows that the difference
between the two sides of the Hirzebruch-type inequality is exactly twice
the defect from the maximal Milnor number. More precisely, if
\[
\operatorname{def}_{\mathrm{H}}(C)
:=
a_{1}+d_{4}+2
-
\left(
\sum_{t\geq 3}(t-2)a_{2t-1}
+
\sum_{t\geq 3}(t-2)d_{2t+2}
+
\frac{3}{2}\ell
\right)
\]
and
\[
\operatorname{def}_{\mu}(C)
:=
3m^{2}-3m+1-\mu(C),
\]
then
\[
\operatorname{def}_{\mathrm{H}}(C)
=
2\operatorname{def}_{\mu}(C).
\]
In particular, the inequality in Theorem~5.2 is sharp precisely for
maximising curves. Within the class of algebraic conic-line arrangements
with ADE singularities, sharpness is determined by the numerical
singularity data and forces freeness with the uniquely determined exponents
\((m-1,m)\).
\end{remark}

\section{Divisible arrangements}\label{s:divisible}

Recall that in Section~\ref{ss:multisequence} we defined the multiplicity multisequence of a combinatorial conic-line arrangement.
The following is the key definition of this section.

\begin{definition}
    Let $q$ be an integer. The combinatorial conic-line arrangement $\mathcal{A}$ is $q$-\emph{divisible} if $q$ divides both $\deg\mathcal{A}$ and each entry $m_j$ in the multiplicity multisequence $(m_j)$ of $\mathcal{A}$.
    We say that it is \emph{even} if it is $2$-divisible.
\end{definition}

For instance, if the arrangement $\mathcal{A}$ has even degree and only singularities of type $A_{2t-1}$ or ordinary $2s$-fold points, then it is even. If it only has triple points (or singularities of type $D_4$) and degree divisible by 3, then it is 3-divisible.

\begin{theorem}
\label{prop:q-divisible}
Let $\mathcal{A}$ be a the combinatorial conic-line arrangement that is $q$-divisible for a prime power
$q$.
Let $k$ be the number of conics and $\ell$ be the number of lines of $\mathcal{A}$.
Suppose also that $\mathcal{A}$ is neither a pencil of conics nor a pencil of lines.
Then, if $\mathcal{A}$ is topologically realisable,
\begin{equation}\label{eq:q-divisible}
 \sum_{j=1}^N\left(m_j-\frac{6q}{q+1}\right)\leq 4k+\ell.
\end{equation}
For $q=2$, i.e. for even arrangements, this becomes
\begin{equation}
\label{eq:q2-divisible}
\sum_{j=1}^N(m_j-4)\leq 4k + \ell.
\end{equation}
\end{theorem}

The proof of the theorem follows rather closely that of~\cite[Theorem~3.1]{AcetoGolla}

\begin{proof}
We write
\[
h,e_1,\ldots,e_N
\]
for the standard diagonal basis of $H_2(Y;\mathbb{Z})$, where $N$ is the length of the multiplicity multisequence of $\mathcal{A}$ and the $e_j$ are the total-transform exceptional classes.

Let $C$ be a topological realisation of $\mathcal{A}$.
We can resolve each singularity by blowing up the appropriate number of times at each point $p$:
for each component $C_i$ of $C$, we need to blow up at $p$ \emph{along} $C_i$ exactly $\max_{C_j} I(C_i,C_j,p)$ times.

After the resolution, tube the pairwise disjoint strict transforms of the $r=k+\ell$ components together.
In other words, do an embedded connected sum of all the components of the strict transform.
Since all components are spheres, the result is an embedded sphere $S\subset Y$ with class
\begin{equation}\label{eq:S-class}
    [S]=Dh-\sum_{j=1}^Nm_je_j,
\end{equation}
and whose self-intersection is
\begin{equation}\label{eq:S-square}
S^2= 4k + \ell - \sum_{j=1}^Nm_j,
\end{equation}
by combining~\eqref{eq:S-class} and~\eqref{eq:noadjunction}.

The class \eqref{eq:S-class} is divisible by $q$, so there exists a $q$-fold cyclic cover
\[
 X\longrightarrow Y
\]
branched over $S$.
The homology argument of~\cite[Lemma 3.3]{AcetoGolla}, which is similar to what we have used in the proof of Theorem~\ref{thm:hirzebruch-top} above, gives:
\[
 b_1(X)=0,
 \qquad
 b_2(X)=q+qN.
\]
The $G$-signature theorem gives
\[
 \sigma(X)=q(1-N)-\frac{q^2-1}{3q}S^2.
\]
The pullback of the exceptional lattice
$\langle e_1,\ldots,e_N\rangle$ is negative definite of rank $N$: over
$\mathbb Q$, the pullback is injective by the transfer map and its intersection
form is multiplied by $q$.  Hence $b_2^-(X)\geq N$.

Using $2b_2^-(X)=b_2(X)-\sigma(X)$, the inequality $b_2^-(X)\geq N$ is
equivalent to
\[
 (q+1)S^2+6qN\geq0.
\]
Substituting \eqref{eq:S-square} and rearranging gives
\eqref{eq:q-divisible}.
\end{proof}

As an immediate corollary, we get Theorem~\ref{thm:divisible-weak} from the introduction.

\begin{proof}[Proof of Theorem~\ref{thm:divisible-weak}]
    Since $\mathcal{A}$ is $q$-divisible, we can apply Theorem~\ref{prop:q-divisible} above. Equation~\eqref{eq:q-divisible} can now be rewritten as:
    \[
    \sum_m (m-4)t_m = \sum_{j=1}^N(m_j-4)\leq 4k + \ell.\qedhere
    \]
\end{proof}

\section{Odd curve configurations}\label{s:odd}

We start the section by defining odd combinatorial conic-line configurations and odd (topological or smooth) curve configuration.

\begin{definition}
    A combinatorial conic-line arrangement $\mathcal{A}$ is said to be \emph{odd} if $\deg \mathcal{A}$ is odd and each $m_j$ in the multiplicity multisequence $(m_j)$ of $\mathcal{A}$ is odd.

    Likewise, a (topological or smooth) curve configuration $C$ in $\CP$ is \emph{odd} if $\deg C$ is odd and all entries in its multiplicity multisequence are odd.
\end{definition}

While the terminology might suggest a dichotomy between odd and even conic-line arrangement, it is quite clear that there are conic-line arrangements that are neither even nor odd---for instance three lines meeting at three nodes form a configuration that is neither even nor odd.

Also, the definition immediately implies that any (topological or smooth) realisation of an odd conic-line arrangement is odd.

\begin{theorem}\label{prop:odd}
Let $C$ be a topological realisation of an odd combinatorial conic-line arrangement $\mathcal{A}$, comprising $k$ conics and $\ell$ lines.
Let $(m_1, \dots, m_N)$ be the multiplicity multisequence of $\mathcal{A}$.
Then
\begin{equation}\label{eq:KM}
4k + \ell -1 \equiv \sum_j (m_j-1) \pmod{16}.
\end{equation}

If, furthermore, $C$ is a \emph{smooth} realisation and $\mathcal{A}$ is neither a pencil of lines nor a pencil of conics,
\begin{equation}\label{eq:10/8}
    \sum_j (m_j-9) \le 4k + \ell - 9.
\end{equation}
\end{theorem}

The proof follows very closely the proofs of~\cite[Theorem~3]{G-odd} and of~\cite[Theorem~4.1]{AcetoGolla}.
A generalisation of the first point in the theorem above was hinted at in~\cite[Remark~6]{G-odd}.

\begin{proof}
    We start by proving \eqref{eq:KM}. This follows exactly as in the proof of~\cite[Theorem~3]{G-odd}, as a corollary of the Kervaire--Milnor theorem.
    More precisely, the strict transform of $C$ after the blow-ups consists of a collection of locally-flatly embedded 2-spheres, whose homology class is odd in the basis $h, e_1, \dots, e_N$ of $H_2(Y)$.
    By tubing all these spheres together, we obtain a characteristic 2-sphere $S$ in the sense of Kervaire--Milnor~\cite{KervaireMilnor}.
    By the locally-flat version of the Kervaire--Milnor theorem, due to Lee and Wilczynski~\cite{LeeWilczynski},
    \[
    S \cdot S \equiv \sigma(Y) \pmod{16}.
    \]
    But from Equation~\ref{eq:S-square} and since $\sigma(Y) = 1 - N$, we obtain
    \[
    3D-2r-1 \equiv \sum_j (m_j-1) \pmod{16}.
    \]

    Equation~\eqref{eq:10/8} follows almost exactly as in the proof of~\cite[Theorem~4.1]{AcetoGolla}.
    The key point is that, if $\mathcal{A}$ is not a pencil, then $S\cdot S < 0$, and we can surger out (a neighbourhood of) $S$ and glue in a plumbing of $|S\cdot S| -1$ spheres of self-intersection $+2$, to get a \emph{spin} closed, smooth 4-manifold, $W$.

    The second Betti number and signature of $W$ are easily computed:
    \begin{align*}
        b_2(W) &= b_2(Y) - 1 + |S\cdot S| -1 = N - 1 - 4k - \ell + \sum_j m_j,\\
        \sigma(W) &= \sigma(Y) + 1 + |S\cdot S| -1 = 1 - N - 4k - \ell + \sum_j m_j.
    \end{align*}
    Applying Furuta's 10/8-theorem, asserting that, for such a $W$, $\frac54 |\sigma(W)| \le b_2(W)$, yields the desired inequality.
\end{proof}

As an immediate corollary, we get Theorem~\ref{thm:odd-weak} from the introduction.

\begin{proof}[Proof of Theorem~\ref{thm:odd-weak}]
    Since $\mathcal{A}$ is $q$-divisible, we can apply Theorem~\ref{prop:odd} above. Equation~\eqref{eq:KM} can now be rewritten as:
    \[
    \sum_m (m-1)t_m = \sum_{j=1}^N(m_j-1) \equiv 4k + \ell - 1 \pmod 16.
    \]
    Equation~\eqref{eq:10/8} can be rewritten as:
    \[
    \sum_m (m-9)t_m = \sum_{j=1}^N(m_j-9) \equiv 4k + \ell - 9.\qedhere
    \]
\end{proof}

\section{ADE-transverse conic-line configurations}\label{s:bonus}

We now combine some of the previous situations to give a more comprehensive statement (which, however, does not follow by combining the results from previous sections).

\begin{theorem}\label{t:random}
    Let $C \subset \mathbb{CP}^2$ be a topological realisation of a combinatorial conic-line configuration $\mathcal{A}$, comprising $\ell$ lines and $k$ conics, with $a_{2t-1}$ singularities of type $A_{2t-1}$, $d_{2t+2}$ singularities of type $D_{2t+2}$, and $t_m$ ordinary $m$-fold points for each $m\geq 2$.
    Suppose that $\ell$ is even, and that $\mathcal{A}$ is neither a pencil of lines nor a pencil of conics. Then:
    \[
    a_5 + 2a_7 + d_8 + 2d_{10} + \sum_{\substack{m \ge 6 \\ m \text{ even}}} \frac{m-4}2t_m \le 2k + \frac\ell2 + t_2 + \sum_{m \text{ odd}} \frac{m-1}2t_m.
    \]
\end{theorem}

Note that a singularity of type $A_1$ is an ordinary double point, so $a_1 = t_2$, and a singularity of type $D_4$ is an ordinary triple point, so $d_4 = t_3$.

\begin{proof}
    As observed above, since lines and conics are non-singular, the only ${\rm ADE}$ singularities that can appear of them are of types $A_{2t-1}$ and $D_{2t+2}$ with $t \le 4$.
    
    Blow up the $\mathbb{CP}^2$ in the following way.
    \begin{itemize}
        \item[(a)] At each singularity of type $A_{2t-1}$, blow up $t$ times to separate the two branches of the singularity. Call $E_A$ the collection of these $\sum_t a_{2t-1}$ divisors.
        
        \item[(d)] At each singularity of type $D_{2t+2}$, first blow up until the normal crossing divisor resolution. This requires $t$ blow-ups. The configuration now contains $t+2$ nodes. Blow up once at each of these nodes. Call $R_D$ the collection of strict transforms of the exceptional divisors of the first $t$ blow-ups, for \emph{all} of these singularity, and $E_D$ the union of all exceptional divisors of these blow-ups. (So $R_D$ comprises $\sum_t t d_{2t+2}$ components and $E_D$ comprises $\sum_t (2t+2)d_{2t+2}$ components.)
        
        \item[(e)] At each ordinary $m$-fold point with $m$ even, blow up once. Call $E_E$ the collection of the $\sum_{m \ge 4, \text{ even}} t_m$ exceptional divisors. (Note that this is consistent with what happens at $A_1$-singularities in point (a) above.)
        
        \item[(o)] At each ordinary $m$-fold point with $m$ odd, blow up once, and then once at each node of the configuration (so $m$ more blow-ups). Call $R_O$ the strict transforms of the first blow-ups and $E_O$ the whole collection of the $m+1$ exceptional divisors of the blow-ups, across all of the singular points of this type. To clarify: $R_O$ comprises $\sum_{m \ge 5 \text{ odd}} t_m$ divisors, and $E_O$ comprises $\sum_{m \ge 5 \text{ odd}} (m+1)t_m$ components.
        (Note that this is consistent with what happens at $D_4$-singularities in point (a) above.)
    \end{itemize}

    Call $Y$ the corresponding blow-up of $\mathbb{CP}^2$, which has:
    \begin{align*}
    \chi(Y) &= 3 + \sum_t t a_{2t-1} + \sum_t (2t+2)d_{2t+2} + \sum_{\substack{m \ge 4 \\ m \text{ even}}} t_m + \sum_{\substack{m \ge 5 \\ m \text{ odd}}} (m+1)t_m,\\
    \sigma(Y) &= 1 - \sum_t t a_{2t-1} - \sum_t (2t+2)d_{2t+2} - \sum_{\substack{m \ge 4 \\ m \text{ even}}} t_m - \sum_{\substack{m \ge 5 \\m \text{ odd}}} (m+1)t_m,
    \end{align*}


     The strict transform $\tilde{C}$ of $C$ in $Y$ is a union of pairwise-disjoint embedded 2-spheres. Let $R_C$ be the internal connected sum of all of these spheres, which we can assume to be disjoint from $R_D$ and $R_O$: that is to say, $R_C$ is an embedded 2-sphere whose homology class is the same as the sum of the homology classes of the components of $\tilde{C}$, and whose self-intersection is the sum of the self-intersections of the components of $\tilde{C}$. In particular, since $C$ is neither a pencil of conics nor a pencil of lines, $R_C\cdot R_C < 0$.
     Finally, let $R = R_C \cup R_D \cup R_O$. By construction $R$ is a collection of $1 + \sum_t td_{2t+2} + \sum_{m \ge 5,\, m \text{ odd}} t_m$ pairwise-disjoint spheres, whose total self-intersection is:
    \[
    \ell + 4k -\left(\sum_t 2ta_{2t-1} + \sum_t (6t+4)d_{2t+2} + \sum_{\substack{m \ge 4 \\ m \text{ even}}} mt_m + \sum_{\substack{m \ge 5 \\ m \text{ odd}}} (3m+1)t_m\right)
    \]

    We claim that there is a double cover $X \to Y$ ramified over $R$.
    This follows from the fact that $\ell$ is even, and from an easy, if tedious, local computation at each singularity, depending on the case (a), (d), (e), or (o) above.

    As above, using inclusion-exclusion for the Euler characteristics and the G-signature theorem for the signature, we can compute:
    \begin{align*}
    \chi(X) &= 2\chi(Y) - \chi(R) = \\
    & = 4 + \sum_t 2ta_{2t-1} + \sum_t (2t+4)d_{2t+2} + \sum_{\substack{m \ge 4 \\ m \text{ even}}} 2t_m + \sum_{\substack{m \ge 5 \\ m \text{ odd}}} 2mt_m
    \\
    \sigma(X) &= 2\sigma(Y) - \frac12 R\cdot R\\
    &= 2 - 2k - \frac\ell2 - \sum_t t a_{2t-1} -\sum_t(t+2) d_{2t+2} + \sum_{\substack{m \ge 4 \\ m \text{ even}}} \frac{m-4}2t_m - \sum_{\substack{m \ge 5 \\ m \text{ odd}}} \frac{m+3}2 t_m.
    \end{align*}

    We claim that $b_1(X) = 0$.
    Since $X$ is the double cover of $Y$ (which has trivial $H_1$) ramified over a disjoint union of spheres of negative self-intersection, it suffices to show that there is a unique linear relation modulo 2 among the components of $R$ (see, for instance,~\cite[Section~3]{RL}).
    To show this, in turn, it suffices to show that the components of $R_D \cup R_O$ are linearly independent modulo 2: this is detected by the components of $E_D$ and $E_O$.
    
    Combining all of this together, we have
    \begin{align*}
    2b_2^-(X) &= b_2(X) - \sigma(X) = \chi(X)-2-\sigma(X)\\
    &= 2k + \frac\ell2 + \sum_t 3t a_{2t-1} + \sum_t (3t+6) d_{2t+2} - \sum_{\substack{m \ge 4 \\ m \text{ even}}} \frac{m-8}2t_m + \sum_{\substack{m \ge 5 \\ m \text{ odd}}} \frac{5m+3}2 t_m.
    \end{align*}

    We can also rewrite this as follows.
    \begin{equation}\label{e:b2m-ADE+ordinary}
        2k + \frac\ell2 + 6a_3 + 9a_5 + 12a_7 + 12d_6 + 15d_8 + 18d_{10} - \sum_{m \text{ even}} \frac{m-8}2t_m + \sum_{m \text{ odd}} \frac{5m+3}2 t_m.
    \end{equation}

    We now observe that $X$ contains:
    \begin{itemize}
        \item[(a)] an $A_{2t-1}$-configuration for each singularity of $C$ of type $A_{2t-1}$;
        \item[(d)] a $D_{2t+2}$-configuration for each singularity of $C$ of type $D_{2t+2}$;
        \item[(e)] a curve of self-intersection $-2$ for each ordinary singularity of $C$ of even multiplicity;
        \item[(o)] $m+1$ curves arranged as in the figure below for each ordinary singularity of $C$ of multiplicity $m$, for $m$ odd.
    \end{itemize}

    We now write the inequality:
    \[
    \sum_t (2t-1)a_{2t-1} + \sum_t (2t+2)d_{2t+2} + \sum_{m \ge 4 \text{ even}} t_m + \sum_{m \ge 5 \text{ odd}} (m+1)t_m \le b_2^-(X),
    \]
    or equivalently:
    \begin{equation}\label{e:countnegativedivisors}
    6a_3 + 10a_5 + 14a_7 + 12d_6 + 16d_8 + 20d_{10} + \sum_{m \text{ even}} 2t_m + \sum_{m \text{ odd}} 2(m+1) t_m \le 2b_2^-(X).
    \end{equation}
    Plugging in the value of $b_2^-(X)$ from~\eqref{e:b2m-ADE+ordinary} above and rearranging, we get:
    \[
    a_5 + 2a_7 + d_8 + 2d_{10} + \sum_{\substack{m \ge 6 \\ m \text{ even}}} \frac{m-4}2t_m \le 2k + \frac\ell2 + t_2 + \sum_{m \text{ odd}} \frac{m-1}2t_m.\qedhere
    \]
\end{proof}

\begin{remark}
    Something stronger can be said on a case-by-case bases. In the proof above, one can avoid taking the internal connected sum of the components of $\tilde{C}$ and work with the whole strict transform instead of working with $R_C$.
    
    The two changes are the computation of the Euler characteristic and the first Betti number of the cover.
    To be more precise, let $b+1$ be the dimension of the space of linear relations modulo 2 among the components of $R' := \tilde{C} \cup R_D \cup R_O$ (see~\cite{RL}).
    Note that $b+1 \le k+\ell$, since the components of $R_D \cup R_O$ are linearly independent modulo 2.
    Let $X'$ be the double cover of $Y$ ramified over $R'$.
    Then
    \begin{align*}
    \chi(X') &= 6 - 2k - 2\ell + \sum_t 2ta_{2t-1} + \sum_t (2t+4)d_{2t+2} + \sum_{\substack{m \ge 4 \\ m \text{ even}}} 2t_m + \sum_{\substack{m \ge 5 \\ m \text{ odd}}} 2mt_m,\\
    b_1(X') &= b \le k + \ell - 1.
    \end{align*}
    The inequality one gets at the end of the day is:
    \[
    \frac32\ell + a_5 + 2a_7 + d_8 + 2d_{10} + \sum_{\substack{m \ge 6 \\ m \text{ even}}} \frac{m-4}2t_m \le 2 + 2b + t_2 + \sum_{m \text{ odd}} \frac{m-1}2t_m.
    \]
    The inequality in Theorem~\ref{t:random} covers the worst-case scenario, when $b + 1 = k+\ell$.
\end{remark}

\section{On curve singularities, double covers, and resolutions}\label{s:suspension}

In the proof of Theorem~\ref{t:random}, we also proved that if a plane curve $C\subset \CP$ of even degree has a singularity $(C,p)$ of type $A_{2k+1}$ or $D_{2k}$, then there exists a (possibly non-minimal) resolution $Y \to \CP$ of $(C,p)$ and a non-singular curve $\tilde{C} \subset R \subset \overline{C}$ such that:
\begin{itemize}
\item there exists a double cover $X \to Y$ ramified over $R$, and
\item $X$ contains a (possibly non-minimal) resolution of the corresponding $A_{2k+1}$- or $D_{2k}$-singularity.
\end{itemize}

In this section we prove that the same statement holds for \emph{every} singularity. The procedure can be made very explicit for the other simple singularities; namely for types $A_{2k}$, $D_{2k+1}$, $E_6$, $E_7$, and $E_8$.
We phrase the statements in local terms.

Since the local models for topological or smooth curve arrangements are analytic, we can always glue in the analytic construction and remain the same category (topological or smooth) in our constructions.

\begin{proposition}\label{p:ADE-resolutions}
Let $C \subset B^4 \subset \C^2$ be a complex curve with a unique singularity.
Then there exist a blow-up $Y\to B^4$ with exceptional divisor $E$ and a disjoint union of smooth spheres $R_0 \subset Y$ such that:
\begin{itemize}
    \item there is a double cover of $Y$ ramified over $\tilde{C} \cup R_0$;
    \item contracting the preimage of $E$ in $X$ gives a singular surface $S$ that suspends the singularity of $C$.
\end{itemize} 
\end{proposition}

Specialising to a singularity of type ADE, it can be also seen by hand (by implementing the algorithm given in the proof) that the corresponding divisor in $X$ is a resolution of an ADE singularity.

\begin{proof}
    For convenience, suppose that the singularity of $C$ is at $0 \in B^4$.
    Let $\pi_0 \colon Y_0 \to B^4$ be the minimal SNC-divisor resolution of $C$.
    That is to say $\pi_0^{-1}(C)$ is a simple normal crossing divisor (in particular, $\tilde{C}_0$ is non-singular).

    Suppose now that $D \subset \Delta$ is a germ of curve of multiplicity $m$. Note that $H_1(\Delta\setminus D) = \Z^c$, generated by the meridians of the components $D_1, \dots, D_c$ of $D$, and therefore there is a unique homomorphism $\pi_1(\Delta\setminus D) \to \Z/2\Z$ sending each such meridian to a non-zero element.
    This homomorphism corresponds to the (unique) double cover of $\Delta$ ramified over $D$.
    (In terms of equations, if $D$ is defined by the equation $f \in \C[x,y]$, the double cover is defined by $z^2 + f(x,y)$ as a subset of $\Delta\times\C$.)
    If we blow up $\Delta$ at the singularity of $D$, we obtain $\overline{D} = \tilde{D}\cup E \subset \tilde{\Delta}$.
    Of course, the projection $\tilde{\Delta} \to \Delta$ restricts to a diffeomorphism $\tilde{\Delta} \setminus \overline{D} = \Delta \setminus D$.
    We can view $H_1(\tilde{\Delta}\setminus \overline{D})$ as being generated by the meridians of $\tilde{D}$ and $E$. The main observation is that, under the identification $H_1(\tilde{\Delta}\setminus \overline{D}) \cong H_1(\Delta\setminus D) = \Z^c$ given by the projection, the meridian of $E$ is sent to $m \in \Z/2\Z$.

    The argument above gives a procedure to determine the ramification data of the double cover of $Y_0$ which is the pull-back of the double cover of $\Delta$ ramified over $C$. This double cover has singularities precisely where two divisors of $\tilde{C}_0$ whose meridians are sent to $1 \in \Z/2\Z$ intersect in $Y_0$.

    Blowing up once at these points we obtain $Y_0$ and the required ramification set $R$, which comprises all divisors in $\overline{C}$ whose meridians are sent to $1$.

    Note that there is a degree-2 map $p\colon X \to S$, defined as follows. Let $E \subset X$ be the preimage of $\pi^{-1}(0) \subset Y$.
    By construction, $X \setminus E$ double covers $\Delta \setminus \{0\}$ with ramification $C \setminus \{0\}$.
    Therefore, there is a map $p_0\colon X \setminus E \to S \setminus \{0\}$ (with a slight abuse of notation, we call $0 \in S$ the preimage of $0 \in \Delta$ by the double cover) and one can pull back the analytic structure on $S \setminus \{0\}$ to get an analytic map $X \setminus E \to S \setminus 0$.
    Note that if we extend $p_0$ by sending $E$ to $0 \in S$, we obtain a \emph{continuous} map $p \colon X \to S$.

    We now claim that $X$, the double cover of $Y$ ramified over $R$, is a resolution of $S$, the double cover of $\Delta$ ramified over $C$.
    This is a consequence of the Riemann Extension Theorem and of the negativity of contracting divisors.
    
    Since $p_0$ extends to the continuous map $p \colon X \to \Delta$, $p_0$ is in particular bounded along $E$, and by the Riemann Extension Theorem (see~\cite[Section~1.C]{GunningRossi}) it extends to an analytic function $X \to S$. Since this extension has to be continuous and the continuous extension is unique, $p_0$ is analytic.
    Now, since $X$ and $S$ are analytic spaces, the continuation is also analytic.
    We now have a proper analytic map $X\to S$ which is a biholomorphism away from $p^{-1}(0)$, which is therefore a resolution of the singularity of $S$.
    By uniqueness of the minimal resolution for surfaces (see~\cite[Section 2.4]{Nemethi-book}), we conclude that $E$ is a blow-up of the minimal resolution.
\end{proof}

\section*{Tool use}
No LLM was used during the writing of this paper.

\section*{Funding}
Piotr Pokora is supported by the National Science Centre (Poland) Sonata Bis Grant \textbf{2023/50/E/ST1/00025}. For the purpose of Open Access, the author has applied a CC-BY public copyright license to any Author Accepted Manuscript (AAM) version arising from this submission.


\begin{thebibliography}{99}

\bibitem[AG26]{AcetoGolla}
P.~Aceto and M.~Golla, Topological line arrangements with high multiplicities. Preprint: 
\url{http://arXiv.org/abs/2607.09398} (2026).

\bibitem[AMP26]{AcetoMcCoyPark}
P.~Aceto, D.~McCoy, J.~Park, Smooth realizations of line configurations. Preprint: \url{http://arXiv.org/abs/2607.09439} (2026).

\bibitem[ACT08]{Artaletal-Survey}
E.~Artal Bartolo, J.I.~Cogolludo, H-o.~Tokunuga, 
A survey on Zariski pairs. Konno, Kazuhiro (ed.) et al., Algebraic geometry in East Asia—Hanoi 2005. Proceedings of the 2nd international conference on algebraic geometry in East Asia, Hanoi, Vietnam, October 10–14, 2005. Tokyo: Mathematical Society of Japan. Advanced Studies in Pure Mathematics \textbf{50}, 1--100 (2008).

\bibitem[AS68]{AtiyahSinger}
M.~F. Atiyah and I.~M. Singer, The index of elliptic operators: III. \textit{Ann. of Math. (2)} \textbf{87}, 546--604 (1968).

\bibitem[BFLL18]{BaaderFellerLewarkLiechti}
S.~Baader, P.~Feller, L.~Lewark, L.~Liechti, On the topological 4-genus of torus knots.  Trans. Am. Math. Soc. \textbf{370}(4), 2639--2656 (2018).

\bibitem[Boj03]{Bojanowski}
R. Bojanowski, Zastosowania Uog\'olnionej Nier\'owno\'sci Bogomolova-Miyaoka-Yau, Master Thesis (in Polish), 2003. Available at \url{http://www.mimuw.edu.pl/~alan/postscript/bojanowski.ps}.

\bibitem[Bri68]{Brieskorn}
E.~Brieskorn, Die Aufl\"osung der rationalen Singularit\"aten holomorpher Abbildungen. \textit{Math. Ann.} \textbf{178}, 255--270 (1968).

\bibitem[BJP26]{BromboszczJaroslawskiPokora}
A.~Bromboszcz, B.~Jaros\l awski, and P.~Pokora, On plus--one generated arrangements of plane conics. \textit{Geom. Dedicata} \textbf{220}, Paper No.~10 (2026).

\bibitem[DP23]{DimcaPokoraMaximizing}
A. Dimca and P. Pokora, Maximizing curves viewed as free curves. \textit{Int. Math. Res. Not. IMRN}, \textbf{2023 (22)}, 19156--19183 (2023).

\bibitem[dPW99]{duP}
A. du Plessis and C. T. C. Wall, Application of the theory of the discriminant to highly singular plane curves. \textit{Math. Proc. Camb. Philos. Soc.} \textbf{126(2)}, 259--266 (1999).

\bibitem[Dur79]{Durfee}
A.~H. Durfee, Fifteen characterizations of rational double points and simple critical points. \textit{Enseign. Math. (2)} \textbf{25}, 131--163 (1979).

\bibitem[Fur01]{Furuta}
M. Furuta, Monopole equation and the $11/8$-conjecture. \textit{Math. Res. Lett.} \textbf{8}, 279--291 (2001).

\bibitem[Gil81]{Gilmer}
P. Gilmer, Configurations of surfaces in 4-manifolds. \textit{Trans. Amer. Math. Soc.} \textbf{264}(2), 353--380 (1981).

\bibitem[Gol24]{G-odd}
M. Golla, On line arrangements with odd multiplicities. \textit{New York J. Math.} \textbf{30}, 1643--1647 (2024).

\bibitem[GR65]{GunningRossi}
R. C. Gunning and H. Rossi, Analytic Functions of Several Complex Variables. Modern Analysis. Englewood Cliffs, NJ: Prentice-Hall (1965).

\bibitem[Hir82]{Hirzebruch1}
F. Hirzebruch, Some examples of algebraic surfaces. \textit{Contemp. Math}. \textbf{9}: 55--71 (1982).

\bibitem[Hir83]{Hirzebruch}
F.~Hirzebruch, Arrangements of lines and algebraic surfaces. in: Arithmetic and Geometry, Vol.~II, Progr. Math. 36, Birkh\"auser, Boston, 1983, 113--140.

\bibitem[HSz71]{HsiangSzczarba}
W.~C. Hsiang and R.~H. Szczarba, On embedding surfaces in four-manifolds. in: Algebraic Topology, Proc. Sympos. Pure Math. \textbf{22}, Amer. Math. Soc., 1971, 97--103.

\bibitem[Lan03]{Langer}
A.~Langer, Logarithmic orbifold Euler numbers of surfaces with applications. 
Proc. Lond. Math. Soc., III. Ser. \textbf{86(2)}, 358--396 (2003).

\bibitem[LW90]{LeeWilczynski}
R.~Lee and D.~M.~Wilczynski, Locally flat 2-spheres in simply connected 4-manifolds. Comment. Math. Helv. \textbf{65}(3), 388--412 (1990).

\bibitem[KM61]{KervaireMilnor}
M.~A.~Kervaire and J.~W.~Milnor, On 2-spheres in 4-manifolds. Proc. Natl. Acad. Sci. USA \textbf{47}, 1651--1657 (1961).

\bibitem[Kob90]{Kobayashi}
R.~Kobayashi, An application of K\"ahler--Einstein metrics to singularities of plane curves. in: Recent Topics in Differential and Analytic Geometry,
Adv. Stud. Pure Math. \textbf{18-I}, 321--326 (1990).

\bibitem[KM94]{KronheimerMrowka}
P.~B.~Kronheimer and T.~S.~Mrowka, The genus of embedded surfaces in the projective plane. Math. Res. Lett. \textbf{1}(6), 797--808 (1994).


\bibitem[N\'em22]{Nemethi-book}
A. N\'emethi, Normal surface singularities, Ergebnisse der Mathematik und ihrer Grenzgebiete, 3. Folge, vol. \textbf{74}, Springer-Verlag, Berlin (2022).

\bibitem[Per82]{Persson}
U. Persson, Horikawa surfaces with maximal Picard numbers. \textit{Math. Ann.} \textbf{259}: 287--312 (1982).

\bibitem[Pok17]{Pokora-conicline-BMY}
P. Pokora, The orbifold Langer--Miyaoka--Yau inequality and Hirzebruch-type inequalities. Electron. Res. Announc. Math. Sci. \textbf{24}, 21--27 (2017).


\bibitem[Pok21]{Pokora}
P. Pokora, Hirzebruch-type inequalities viewed as tools in combinatorics. 
\textit{Electron. J. Comb.} \textbf{28}(1): Paper No. P1.9, 22 p. (2021).


\bibitem[RS19]{RL}
D. Ruberman and L. Starkston, Topological realizations of line arrangements. \textit{Int. Math. Res. Not.} \textbf{2019(8)}: 2295--2331 (2019).

\end{thebibliography}
\end{document}